\documentclass[12pt]{amsart}
\usepackage[margin=3.1 cm]{geometry}
\usepackage{amsmath}
\usepackage{amsfonts}
\usepackage{amssymb}
\usepackage{xcolor}

\renewenvironment{proof}{{\bfseries Proof.}}{\qed}
\numberwithin{equation}{section} 
\theoremstyle{theorem}
\newtheorem{theorem}{Theorem}[section] 
\newtheorem{pro}[theorem]{Proposition} 
\newtheorem{cor}[theorem]{Corollary} 
\newtheorem{lemma}[theorem]{Lemma} 

\theoremstyle{definition}
\newtheorem{defn}[theorem]{Definition}

\newcommand{\n}{\newline}

\newcommand{\tb}{\textbf}

\newcommand{\ap}{\alpha_{p}}
\newcommand{\imp}{\Rightarrow}
\newcommand{\z}{\mathbb{Z}}

\newcommand{\R}{\mathbb{R}}
\newcommand{\ps}{S'}
\newcommand{\g}{\gamma_{p}^{m}}
\newcommand{\pc}{$\mathcal C'$}

\newcommand{\secref}[1]{Section~\ref{#1}}
\newcommand{\thmref}[1]{Theorem~\ref{#1}}
\newcommand{\lemref}[1]{Lemma~\ref{#1}}

\newcommand{\propref}[1]{Proposition~\ref{#1}}
\newcommand{\corref}[1]{Corollary~\ref{#1}}

\newcommand{\defref}[1]{Definition~(\ref{#1})}

\numberwithin{equation}{section} 
\def \h{{\bf H}} 

\usepackage[backref]{hyperref}

\begin{document}
	\title{Reciprocity in the Hecke Groups}
	\author{Debattam Das \and Krishnendu Gongopadhyay}
	\address{Indian Institute of Science Education and Research (IISER) Mohali,
		Knowledge City,  Sector 81, S.A.S. Nagar 140306, Punjab, India}
	\email{debattam123@gmail.com}
	\email{krishnendug@gmail.com, krishnendu@iisermohali.ac.in}
	
	\maketitle	
	\begin{abstract}
		An element $g$ in a group $G$ is called \emph{reciprocal} if there exists $h \in G$ such that $g^{-1}=hgh^{-1}$.   The reciprocal elements are also known as `real elements' or `reversible elements' in the literature. In this paper, we consider the Hecke groups which are Fuchsian groups of first kind and generalization of the modular group. We have classified and parametrized the reciprocal classes in the Hecke groups. This generalizes a result by Sarnak on the reciprocal elements in the modular group. 
	\end{abstract}

	\vspace{0.5 in}

	\section{Introduction} 
	An element $g$ in a group $G$ is called \emph{reciprocal} if there exists $h \in G$ such that $g^{-1}=hgh^{-1}$. The reciprocal elements are also known as `real elements' and `reversible elements'.  We will call the conjugating element $h$ as a \emph{reciprocator} of $g$. A reciprocator is also known as a \emph{reverser} or a \emph{reversing symmetry}. For a reciprocal element $g$, the set of all reciprocators is an index two extension of the centralizer of $g$, see \cite{BR}. 
	We call an element $g$ in $G$ as  \emph{strongly reciprocal} if $h$ is an involution, i.e., $h^2=1$. Equivalently, an element $g$ is strongly reciprocal if it is a product of (at most) two involutions. Such elements are also known as `bi-reflectional', `strongly real', or `strongly reversible'. A strongly reciprocal element is reciprocal,  but the converse is not generally true. The reciprocal and strongly reciprocal elements appear at several places in the literature, e.g.  \cite{FZ}, \cite{l}, \cite{lr}, \cite{os}, \cite{Sa} and the references therein.  In a group $G$, given an infinite order strongly reciprocal element $g$ and its reciprocator (or reverser) $r$, the group $\langle g, r\rangle$ is an infinite dihedral subgroup. When $G$ is a Fuchsian group, the conjugacy classes of maximal infinite dihedral subgroups are in one-to-one correspondence with primitive reciprocal geodesics in the corresponding surface. Here, a geodesic is reciprocal if it is equivalent to itself with orientation reversed.   

 Sarnak \cite{Sa}  parametrized and counted reciprocal classes for the modular group. Reciprocity in the modular group ${\rm PSL}(2, \z)$ has been understood from many other independent viewpoints as well. For a survey see \cite[Chapter 7]{os} and the references therein.  

The Hecke groups are generalization of the modular group. These are Fuchsian groups of the first kind. Our main aim in this paper is to classify the reciprocal elements in the Hecke groups. It is also a natural problem to ask about reversibility in an arbitrary Fuchsian group. In the following, we first answer this question. 
	
	\begin{theorem}\label{fuchth}
	Let $\Gamma$ be a Fuchsian group. 
	
	\begin{enumerate} 
		\item An element $g$ in $\Gamma$ is reciprocal if and only if it is strongly reciprocal. Further, the reciprocators of a reciprocal element are involutions. 
		
		\item The only possible reciprocal elements in $\Gamma$  are either the hyperbolic elements or the involutions.
	\end{enumerate} 
\end{theorem} 

The above theorem gives an answer to an open problem in \cite[p. 104]{os} when asking it for Fuchsian groups. By the above theorem, any reciprocator of a reciprocal element is an involution.  So, we will not prefix `strongly' anymore before reciprocal elements.  The above result might have been known to experts, eg. \cite{fk}. However, we do not know of any literature where this has been stated explicitly as in the above theorem.  

Now, recall the notion of Hecke groups. 
\begin{defn}
	The Hecke group $ \Gamma_{p} $ is a Fuchsian group which is generated by 
	the maps
	\[\iota: z\mapsto-\dfrac{1}{z}~~~~~ \\ \hbox{ and } \\ ~ ~~~~ \alpha_p: z\mapsto z+ 2\cos {\pi}/{p}, ~ p \geq 3. \]
	
	Geometrically, the Hecke group can be identified with  $ \z_{2} \ast \z_{p}$. 
\end{defn} 
Note that the Hecke group $\Gamma_3$ is the modular group ${\rm PSL}(2, \z)$. In \cite[p.143]{os}, the reciprocity in $\Gamma_p$, for $p \geq 4$, has been listed as an open problem. Our work gives a complete solution to this problem. To state our result, we first 
define the following.
\begin{defn}\label{fpr}
	Let $A$ be a hyperbolic element in ${ \rm PSL}(2,\R) $, represented by the matrix $\tilde A= \begin{bmatrix}
		a&& b\\
		c&& d
	\end{bmatrix}$. 
	We define the fixed-point ratio of $A$ by 
	$$\theta_{A}=\begin{cases}
		\dfrac{b-c}{a-d}, ~~~~~~~~\hbox{  where,  } a\neq d. \\
		\infty,~~~~~~ \hbox{ where, } a=d, ~ b \neq c.\\
		1, ~~~~~~~\hbox{ where,  }a=d ,~ b=c.	
	\end{cases}
	$$ 
\end{defn} 

The fixed-point ratio is an invariant of the cyclic group generated by $A$. 
When $a \neq d$ or, $a=d, ~ b \neq c$,  the fixed-point ratio is given by $ \theta_{A}=-\dfrac{1+u_{1}u_{2}}{u_{1}+u_{2}}$,  where $u_1$ and $u_2$ are the fixed-points of $A$. We would like to note that the degenerate case $a=d$ may arise in Hecke groups.  In this case,  any non-zero powers of $A$ possess `symmetric' fixed points of the form $\{p, -p\}$, where $p\in \mathbb{R}$. In the particular case when $a=d$ and $b=c$, the fixed point set is $\{-1,1\}$, and the fixed-point ratio is defined to be one of the fixed points.  Such elements assume a special role in the Hecke group $\Gamma_p$ when $p$ is even.  

The following theorem classifies reciprocal elements in the Hecke groups. It also generalizes the corresponding result for the modular group $\Gamma_3$ in  \cite{Sa},  \cite[Proposition 7.30]{os}.  

\begin{theorem}\label{thm}
	Let $\Gamma_p$ be the Hecke group for $p \geq 3$. Let $g \in \Gamma_p$. Then the following are equivalent.
	\begin{enumerate}
		\item $g$ is reciprocal. 
		\item For $p$ odd,  $g$ is either an involution or conjugate to a hyperbolic element $h$ such that the fixed-point ratio of $h$ is zero; equivalently, lifts of $h$ in $\rm SL(2,\R) $ are symmetric matrices. 
		
		For $p$ even, $g$ is either an involution or conjugate to a hyperbolic element $h$ such that the fixed-point ratio of $h$ is either zero, $ 1 $,  or,   $\cos {\pi}/{p}$.
	\end{enumerate}  
\end{theorem}

The next question about reciprocity in the Hecke groups $\Gamma_p$ would be to parametrize the reciprocal classes.   For this,  we will consider only the hyperbolic reciprocal classes.  Given a reciprocal element $A$, one needs to know the number of elements in the reciprocal class of $A$ with the same fixed-point ratio $\theta_A$. The following theorem answers this question. Sarnak \cite{Sa} proved a version of this theorem for the modular group. Our proof of the following theorem follows similar ideas for the Hecke group. To state the theorem, we first define the following. 

\begin{defn}
	A reciprocal element $ M $ in $\Gamma_{p}$ is said to be \textbf{$p$-reciprocal} if the fixed point ratio $\theta_{M}$ is either  $\cos \pi/p $ or,  $ 1 $. Equivalently, \tb{$p$-reciprocal } elements have  $ \g $ as reciprocator.
\end{defn}
\begin{defn}
	A reciprocal element $ M $ in $ \Gamma_{p} $ is said to be \tb{symmetric $p$-reciprocal} if the fixed point ratio $\theta_{M}$ is $1$. Equivalently, \tb{symmetric $p$-reciprocal} elements have reciprocators $ \g $, as well as $ \iota $.
\end{defn}
\begin{theorem}\label{sym}
	For any reciprocal hyperbolic element in $ \Gamma_{p} $, the following possibilities can occur :
	\begin{enumerate}
		\item If p is odd, every reciprocal class has exactly four symmetric elements.
		\item If p is even.
		\begin{enumerate}
			\item Suppose in the reciprocal class, the reciprocators are conjugate to $ \iota $. Then, there are exactly four symmetric elements in that class.
			\item Suppose in the reciprocal class, the reciprocators conjugate to $\gamma_{p}^{m} $. Then, there are exactly $ 2p $ $p$-reciprocal elements. 
			\item Suppose in the reciprocal class, each reciprocal element has two types of reciprocators, one type of reciprocators conjugate to $ \iota $ and another type of reciprocator conjugate to $ \gamma_{p}^{m} $.
			\begin{enumerate}
				
				\item If the reciprocal class does not contain any non-zero power of $ \iota \g $, then the class has exactly two symmetric elements and $ p $ $p$-reciprocal elements.
				
				\item If the reciprocal class contain any non-zero power of $ \iota \g$, then the class has exactly two symmetric p-reciprocal elements and $p-2$  $p$-reciprocal elements.			\end{enumerate}
		\end{enumerate}
	\end{enumerate}
\end{theorem}

The above theorem parametrizes the reciprocal classes as follows.  Let $\rho$ be the set of the conjugacy classes of reciprocal elements in the Hecke group $\Gamma_p$ other than the involution classes. 

\[ \mathcal{C}^{\prime}=\{(a,b,c,t)\in \mathbb{Z}[\lambda_{p}]^{4}~|~a,c,t>0,\  d=4ac+b^{2}, \ t^{2}-d=4,\ ((t-b)/2,c)_{p}=((t+b)/2,a)_{p}=1,\] \[ a=c,\hbox{ or } \dfrac{c-a}{b}=\cos \pi/p \}.\]
Define, $\phi^{\prime}:\mathcal{C}^{\prime}\mapsto \rho $: 
\begin{equation}\label{main}
	(a,b,c,t)\longmapsto \left\{ \begin{bmatrix}
		(t-b)/2&& a\\
		c&& (t+b)/2
	\end{bmatrix}\right\}_{\Gamma_{p}}.\end{equation} 
\begin{cor}\label{main1}
	There is a parametrization of the reciprocal classes in $\Gamma_p$ as follows.
	\begin{enumerate}
		\item If $ p $ is odd. Then, every pre-image of $\phi^{\prime}$ of a reciprocal class has two elements.
		\item If $ p=2m $ for $ m\in\z $. Then one of the following cases will occur:
		\begin{enumerate}
			\item A reciprocal class has reciprocators conjugate to  $\iota$. Then the pre-image of the reciprocal class under ${\phi^{\prime}}$ of contains two elements.
			\item A reciprocal class has reciprocators conjugate to $\gamma_{p}^{m}$. Then the pre-image of the reciprocal class under $\phi^{\prime}$ contains $ p $ elements.
			\item A reciprocal class has two types of reciprocators, some are conjugate to $\iota$, and others conjugate to $\gamma_{p}^{m}$. If the class does not contain any power of $ \iota\g $, then the pre-image of $\phi^{\prime}$ of the reciprocal class contains $ m+1 $ elements. If the reciprocal class contains a power of $ \iota \g $,  then the pre-image of the reciprocal class under  $ \phi' $ contains exactly $ m  $ elements. 
		\end{enumerate}
	\end{enumerate} 
\end{cor}

When $p$ is odd,  we can further modify the parametrization as follows.   Let 
	\[ \mathcal{C}=\{(a,b,t)\in \mathbb{Z}[\lambda_{p}]^{3} ~|~a,t>0, d=4a^{2}+b^{2}, t^{2}-d=4,((t-b)/2,a)_{p}=((t+b)/2,a)_{p}=1 \}, \]
	where $ \lambda_{p}=2\cos\pi/p $.  Define $\phi: \mathcal{C}\mapsto \rho $: \begin{equation}\label{main}
		(a,b,t)\longmapsto \left\{ \begin{bmatrix}
			(t-b)/2&& a\\
			a&& (t+b)/2
		\end{bmatrix}\right\}_{\Gamma_{p}},\end{equation} 
	where, $ \{\gamma\}_{\Gamma_{p}}$ is the conjugacy class of $\gamma$ in $\Gamma_{p}$.

	\begin{cor}
		There is a map $\phi$ from $ \mathcal{C} $ to $\rho$ where every fiber has two elements.
	\end{cor}

As it is known, there are several ways to understand the reciprocity in the modular group. We expect that the same should be true for the Hecke groups as well, and there should be several other methods to understand reciprocity in the Hecke groups. 

\medskip Finally consider the family of Fuchsian groups $\Lambda_{p, q}$, where $ p, q>0$ are distinct integers such that  $1/p+1/q <1/2$.    The group $\Lambda_{p, q}$ is  generated by the maps
 $$\eta_p:  z\longmapsto -\dfrac{1}{z+2\cos \pi/p},  \hbox{ and  }~  \eta_q: z\longmapsto -\dfrac{1}{z+2\cos\pi/q}.$$
The group $\Lambda_{p, q}$ is finitely generated and may be identified with $ \z_{p} \ast \z_{q} $.   Following the proof of \thmref{thm}, it is easy to classify the reciprocal elements in $\Lambda_{p, q}$. We observe  the following in this regard. 
\begin{cor}
If both $p$ and $q$ are odd numbers, then there is no reciprocal element in $\Lambda_{p, q}$.   

If both of $p$ and $q$ are not odd,  then for a element $g$ in $\Lambda_{p, q}$ the following are equivalent.  
\begin{enumerate}
\item $g$ is reciprocal.

\item $g$ is either an involution or conjugate to a hyperbolic element $h$ such that 
$$\theta_{h}=\begin{cases}
		 1 \hbox{ or, } \cos \pi/p ~~~~~~~~\hbox{ if $p$ is even; }\\
		1 \hbox{ or,  } \cos \pi/q ~~~~~~~~\hbox{ if $q$ is even; }\\
		1 \hbox{ or, } \cos \pi/p  \hbox{ or,  } \cos \pi/q ~~~~~~~~\hbox{ if both $p$ and $q$  are even}. 
	\end{cases}
	$$ 

\end{enumerate} 
\end{cor}

\subsubsection*{Structure of the paper} In \secref{prel},  some basic notions are recalled.   \thmref{fuchth} has been proved in \secref{thm1}. The fixed-point ratio and its relationship with reciprocity has been noted in \secref{fpr}.  Proof of \thmref{thm} has been splitted into two sections.   \secref{odd} contains the proof for $p$ odd, and  \secref{even} contains the proof for $p$ even.   \secref{param} is devoted to the proof of \thmref{sym}.  Proof of the subsequent corollary has also been given in this section. 
	
	\subsection*{Acknowledgment} 
	Das acknowledges full support from a UGC-NFSC research fellowship during the course of this work. 
	Gongopadhyay acknowledges partial support from the SERB core research grant CRG/2022/003680. 
	
	\section{Preliminaries} \label{prel}
	Throughout the following, we consider the upper-half space model of the hyperbolic plane $\h^2$. The boundary $\partial \h^2$ may be identified with the circle $\R \cup \{\infty\}$. The group ${\rm SL}(2, \R)$ acts on $\h^2$ by the M\"obius transformations:
	$$\begin{bmatrix} a & b \\ c & d \end{bmatrix}: z \mapsto \frac{az+b}{cz+d}.$$
	The isometry group of $\h^2$ can be identified with ${\rm PSL}(2, \R)$.  For an element $g$ in ${\rm PSL}(2, \R)$, we denote by $\tilde g$ as one of the lifts of $g$ in ${\rm SL}(2, \R)$. And a \textit{symmetric} element in $ \rm PSL(2,\R) $ is an element that is presented by a symmetric matrix in $ \rm SL(2,\R) $. For basic information on Fuchsian groups, we refer to the text \cite{Ka}. 
	
	\medskip Recall that every isometry $g$ of $\h^2$ has a fixed point on $\h^2 \cup \partial \h^2$.  The isometry $g$ is \emph{elliptic} if it has a fixed point on $\h^2$, equivalently if $|{\hbox{trace}}(\tilde g)|<2$; it is \emph{parabolic} if $g$ has a unique fixed point on $\partial \h^2$, equivalently if $|{\hbox{trace}}(\tilde g)|=2$; $g$ is \emph{hyperbolic} if it has exactly two fixed points on $\partial \h^2$, equivalently, if  $|{\hbox{trace}}(\tilde g)|>2$. 
	
	A Fuchsian group is elementary if it has a finite orbit in $\h^2 \cup \partial \h^2$. We note the following result that will be used later. 
	
	\begin{theorem}\label{Ka1}\cite[Theorem 2.4.3]{Ka}
		Any elementary Fuchsian group is either cyclic or is conjugate in ${\rm PSL}(2, \R)$ to a group generated by $g(z)=kz$, $k>1$, and $h(z)=-1/z$. 
	\end{theorem}
	
	\begin{lemma}
		Let $g$ be an element in ${\rm PSL}(2, \R)$. Let $g^{-1}=hgh^{-1}$ for some $h$ in ${\rm PSL}(2, \R)$. Then $h$ keeps the fixed point set $F$ of $g$ invariant, i.e. $h(F)=F$. 
	\end{lemma}
	\begin{proof} 
		Let $x \in F$. Then $hgh^{-1}(h(x))=g^{-1}(h(x)) \Rightarrow g(h(x))=h(x)$. Thus $h(x)$ is also a fixed point of $g$. 
	\end{proof}
	\subsection{ Hecke Groups }
	Let $ p\geq 3 $ be an integer, the Hecke group $ \Gamma_{p} $ is generated by $\iota$ and $\alpha_{p}$. The lifts of $ \iota $ and $\alpha_{p}$ in $\rm SL(2,\R) $ are $S= \begin{bmatrix}
		0&&-1\\
		1&&0
	\end{bmatrix} $ and $T= \begin{bmatrix}
		1&&\lambda_{p}\\
		0&&1
	\end{bmatrix} $ resp. where $\lambda_{p}=2\cos\pi/p$.
	So, we can think Hecke group as a subgroup of $ \rm PSL(2,\z[\lambda_{p}]) $. This group, however, is not itself a discrete group for $ p\geq 4. $ It is possible to know when an element of $ \rm PSL(2,\z[\lambda_{p}]) $ belongs to $\Gamma_{p}$ with the help of the paper \cite{LL}.
	\subsubsection{Pseudo Euclidean algorithm}\cite{LL}
	Let $ a,b~ (\neq 0)\in \z[\lambda_{p}] $. Then there exists an integer $ n$ such that $ a=b(n\lambda_{p})+r   $ where $ -|b\lambda_{p}|/2 \leq r \leq  |b\lambda_{p}|/2$. Now repeat the procedure as follows.
\[a=b(n_{0}\lambda_{p})+r_{1},\]
	\[b=r_{1}(n_{1}\lambda_{p})+r_{2},\]
	\[ \vdots \] 
	\[r_{k-1}=r_{k}(n_{2}\lambda_{p})+r_{k+1},\]
	\[\vdots\]
	\[r_{m-1}=r_{m}(n_{2}\lambda_{p})+r_{m+1}.\]
	If this algorithm terminates at the $ m^{th} $-step, that is $ r_{m+1}=0 $ for some $m \in \mathbb{N} $,
	then, $ (a,b)_{p}:=|r_{m}|. $
	Also,
	\begin{enumerate}
		\item $  (a,b)_{p}=(b,a)_{p} $.
		\item$  (a,b)_{p}= (-a,-b)_{p} = (-a,b)_{p}= (a,-b)_{p}$.	 \end{enumerate}
	\begin{pro} \cite{LL}\label{Li}
		Let $ \tilde A=\begin{bmatrix}
			a&&b\\c&&d
		\end{bmatrix} $ be a lift in $ \rm SL(2,\R) $ of an element $ A $ in $ \rm PSL(2,\z[\lambda_{p}])  $. Then $ A\in \Gamma_{p} $ if and only if $ (a,c)_{p}=(b,d)_{p}=1. $
		
	\end{pro}
	For the proof of the theorem, we refer the proposition 3.7 from \cite{LL}. We use this theorem for the parametrization of the reciprocal classes.
	\section{Proof of \thmref{fuchth} }\label{thm1} 
	
	The \thmref{fuchth} will follow from the following lemmas. 
	
	
	\begin{lemma} \label{pro2}
		The only possible reciprocal elements in a Fuchsian group are either the hyperbolic elements or the involutions. 
	\end{lemma}
	\begin{proof}
		Let $\Gamma$ be a Fuchsian group. To prove this theorem it suffices to show that elliptic elements and parabolic elements are not reciprocal elements. So, we can divide it into two cases.
		
		\textbf{Case 1.} Assume that $\Gamma$ contains an elliptic element. If possible, suppose that an elliptic element $\gamma$ (other than an involution) is reciprocal in $\Gamma$. Then $\exists$ an element $ g \in \Gamma$ such that 
		\[g\gamma g^{-1}=\gamma^{-1}.\]
		Since the elliptic elements in ${\rm PSL}(2, \R)$ have a unique fixed-point in $ \mathbb{H}^{2}  $, 
		$\gamma$ and $\gamma^{-1}$ fix the same point  $x\in \mathbb{H}^{2}$. This implies that $g(x)=x$ and thus,  $ g $, $\gamma$ commutes (from \cite[Theorem 2.3.2]{Ka}). This implies $ \gamma=\gamma^{-1} $, so $\gamma$ is an involution. Then, the involutions are only finite order reciprocal elements in Fuchsian groups.\n
		
		\textbf{Case 2.} For the parabolic case, the proof is very much similar. But they are not finite order elements. Then, this is a contradiction. 
	\end{proof}
	\begin{defn}
		An element $ g $ of a group $ G $, is said to be a \textbf{primitive} element if there does not exist any $ b $ in $ G $ such that $ g=b^{k} $ for some $ k\in \z $ and $ |k|>1 $.
	\end{defn}
	
	In general, primitive elements may not exist in every group. However, for Fuchsian groups, it is necessary for hyperbolic elements to have primitive elements within the group; otherwise, the group would no longer be discrete.
	\begin{lemma}\label{new}
		Let $ g $ be a primitive hyperbolic element in the Fuchsian group $ \Gamma $. Then, $ g $ is reciprocal if and only if $ g^{n} $ is reciprocal for any $ n\in \z - \{0\} $. 
	\end{lemma}
	\begin{proof}
		One part of this lemma is obvious. So, we will show the converse part only. Suppose, 
		$ g^{n} $ is a reciprocal element in $ \Gamma $ with $ n $ a non-zero integer. Then, $\exists$  $ h$ in $ 
		\Gamma $ such that \begin{equation}\label{7}
			h^{-1}g^{n}h=g^{-n}.
		\end{equation} This implies \[ (h^{-1}gh)^{n}=g^{-n}.\] We know that the fixed point set remains invariant under a change of power for the hyperbolic elements. Then, \[ Fix(h^{-1}gh)=Fix(h^{-1}g^{n}h)=Fix(g^{-n})=Fix(g).\]
		So, $ g $ and $ h^{-1}gh $ share the same fixed point set, and $ g $ is the primitive element. Then, from \cite[ Theorem 2.3.5]{Ka} we obtain that \[ h^{-1}gh=g^{k}\] where $k\in \z -\{0\} $.
		This implies $ g $ is not a primitive element. This would be a contradiction if we consider $ |k|>1 $. Therefore, $ k=1  $ or $ -1 $.
		Suppose $ k=1 $. Then, $ h $ lies in the centralizer of $ g $. Then $ g^{2n}=Id $ from \ref{7}, i.e., $ g $ is a finite order element. This implies $ k $ must be $ -1 $. This proves our lemma.
	\end{proof}
	
	\begin{lemma}\label{fuch}
		Let $\Gamma$ be a Fuchsian group. An element $g$ in $\Gamma$ is reciprocal if and only if it is strongly reciprocal. The reciprocators of a reciprocal element are involutions. 
	\end{lemma}
	\begin{proof}
		If $g \in \Gamma$ is strongly reciprocal, then clearly, it is reciprocal. So, we now prove the converse. \n
		{Without loss of generality (from \lemref{new}),  assume that $ g $ is a primitive element in $\Gamma$.}
		Then suppose $g$ is reciprocal. Then there exists an element $h \in \Gamma$ such that $hgh^{-1}=g^{-1}$. Then $h$ keeps the fixed point set of $g$ invariant. In other words,  $\langle g, h \rangle$ would be an elementary subgroup of $\Gamma$.  If $g$ and $h$ have the same fixed point set, then $\langle g, h \rangle$ must be cyclic, using \thmref{Ka1}. Hence $h$ would be a power of $g$. In this case, $g$ would be an involution. 
		
		Suppose $h$ interchanges the fixed points of $g$. Then $h$ has a fixed point which is distinct from the fixed points of $g$. This would imply that $h^2$ fixes three distinct points. This implies $h^2$ must be identity; hence, $h$ is an involution. Thus, this is proved.
	\end{proof}

	\section{The Fixed-point Ratio and Reciprocity } \label{fpr} 
	For a hyperbolic element $A$, let $\theta_A$ be the fixed-point ratio as in \defref{fpr}. 
	\begin{theorem}\label{theta}
		Let $A$ be a hyperbolic element in ${\rm PSL}(2, \R)$.  
		Then $\theta_A = \theta_{A^n}$ for all $n \in \mathbb Z-\{0\}$. In other words, $\theta_A$ is an invariant of the cyclic group generated by $A$. 
	\end{theorem}

	\begin{proof}
Let $\tilde A$ be a lift of $A$ in ${\rm SL}(2, \R)$:  $ \tilde A =\begin{bmatrix}
			a&& b\\
			c&& d
		\end{bmatrix}$. 
		To prove this theorem, first, assume  $ c \neq 0$, then 
		$ \tilde A =\begin{bmatrix}
			a&&b\\
			c&&d
		\end{bmatrix} $ with the absolute value of the trace greater than 2 in $ \rm SL(2,\R) $. Then, $ Fix(A)=\{\dfrac{(a-d)\pm\sqrt{(a+d)^{2}-4}}{2c}\} =\{ u_{1},u_{2}\}$  and also  the equation for the fixed point \[ cz^{2}-(a-d)z-b=0.\] So, $ u_{1}+u_{2}=\dfrac{a-d}{c} $ and $ u_{1}u_{2}=\dfrac{-b}{c}. $ If $ a\neq d $ then, \[\theta_{A}=\dfrac{b-c}{a-d}=\dfrac{b/c-1}{(a-d)/c}\] 
		\[\imp \theta_{A}=-\dfrac{1+u_{1}u_{2}}{u_{1}+u_{2}}.\]
		So $ \theta_{A} $ will not change at the time of iterations.

\medskip  Now, consider $ c= 0 $ then $ Fix(A)=\{\infty, \dfrac{-b}{a-d}\} $ and $\theta_{A}=\dfrac{b}{a-d}$.  
 So, in that case also,  $\theta_{A}$ is invariant at the time of iterations.

		Now suppose $ a=d $ then, $ c $ must be non-zero, otherwise $A$ would be parabolic.  So,  we have $u_{1}+u_{2}=0$, and  $\theta_A=-\frac{1-u_1^2}{u_1+u_2}$, accordingly it is $\infty$.  Since $u_1$ and $u_2$ are also the fixed points of $ A^{n} $ for $ n\in \z -\{0\} $.  So, $ \theta_{A}=\theta_{A^{n}}$. 

In the remaining case, when $a=d$, $b=c$, the fixed points are $1$ and $-1$. We defined $\theta_A=1$. 
Since fixed points do not change under iteration, so $\theta_A$ remains unchanged. 
	\end{proof}\\
	\begin{cor}\label{ttheta}
		Let $ A =\begin{bmatrix}
			a&& b\\
			c&& d
		\end{bmatrix}$ be a matrix with $ |trace(A)|>2 $ in ${ \rm SL}(2,\R) $.  Then, $ \theta_{A}=\theta_{A^{n}} $ for all $ n\in \z -\{0\}$.
	\end{cor}
	\begin{proof}
		The proof of \thmref{theta} follows from the following theorem and the fact that $\theta_A=\theta_{r A}$ for any real number $r$.  
	\end{proof}

	\begin{lemma}\label{co1}
		Let $A$ be a hyperbolic element in ${\rm PSL}(2, \R)$ such that $| \theta_{A}|=\eta < 1 $. Then $A$ is a reciprocal element with a reciprocator 
		$$ {\Phi}=\dfrac{1}{\sqrt{1-\theta_{A}^2}} \begin{bmatrix}
			\theta_{A} && 1\\
			-1&& -\theta_{A}
		\end{bmatrix}. $$
	\end{lemma}
	\begin{proof}
		Let $ \tilde A=\begin{bmatrix}
			a&& b\\
			c&& d
		\end{bmatrix} $. Then \[ a\theta_{A} +c= b+ d\theta_{A}.\]  
		It is easy to see by direct computation that $\Phi A \Phi^{-1}=A^{-1}$ in $ \rm PSL(2,\R) $. This implies, $ A $ is reciprocal with reciprocator $ \Phi $ in $ \rm PSL(2,\R) $. 
	\end{proof}

\medskip We know that every hyperbolic element in $ \rm PSL(2,\R) $ is a product of two involutions. By the previous corollary, we can determine the involutions in the product which are given by 
$$ \dfrac{1}{\sqrt{1-\theta^{2}_{A}}}\begin{bmatrix}
		\theta_{A}&& 1\\
		-1 && -\theta_{A}
	\end{bmatrix} \hbox{  and  }\dfrac{1}{\sqrt{1-\theta^{2}_{A}}}\begin{bmatrix}
		\theta_{A}&& 1\\
		-1 && -\theta_{A}
	\end{bmatrix}A. $$
	
\medskip 
	Observe that for a hyperbolic element $A$ in ${\rm PSL}(2, \R)$ with $\theta_A \neq 1, \infty$, 
	{ \[ 1-\theta_{A}^{2}=1-(\dfrac{1+u_{1}u_{2}}{u_{1}+u_{2}})^{2} \]
		\[ =\dfrac{(u_{1}+u_{2})^{2}- (1+u_{1}u_{2})^{2}}{(u_{1}+u_{2})^{2}}\]
		\[=\dfrac{u_{1}^{2}+u_{2}^{2}-1-u_{1}^{2}u_{2}^{2}}{(u_{1}+u_{2})^{2}}\]
		\[=\dfrac{(u_{1}^{2}-1)-u_{2}^{2}(u_{1}^{2}-1)}{(u_{1}+u_{2})^{2}}\]
		\[=\dfrac{(u_{1}^{2}-1)(1-u_{2}^{2})}{(u_{1}+u_{2})^{2}}.\]
	Thus, for $ |\theta_{A}|<1 $,  one fixed point lie in the interval $ (-1,1) $ and the other lie outside $ (-1,1) $. 

	} 
	\medskip 
	In $ \rm PSL(2,\R) $, we have some hyperbolic elements which are reciprocal in $ \rm PSL(2,\R) $, but the invariant is greater than 1. For example, $B= \begin{bmatrix} 
		3&& 5\\
		1&& 2
	\end{bmatrix} $ gives a hyperbolic element in $ \rm PSL(2,\R) $ that is reciprocal but $ |\theta_{B}|>1 $. 

	For such $B$ whose $ |\theta_{B}|>1 $,  the fixed points are either inside of $ (-1,1) $ or both are outside $ (-1,1) $.   By conjugation,  we can move the fixed points such that one will lie inside $ (-1,1) $. Then, we can do the rest as we have done in \lemref{co1}.  In that case, the reciprocators are conjugate to an element of the form  $ \Phi $. 
	
\medskip 
If we consider $ \theta_{A}=1 $, i.e., $ a=d,c=b $ for the hyperbolic element $ A $, the hyperbolic axis $ I $ of $A$  is the semi-circle of (Euclidean) radius one.  We may obtain the reciprocator $\Phi$ as $ \dfrac{1}{\sin t}\begin{bmatrix}
		\cos t && 1\\ -1 && -\cos t 
	\end{bmatrix} $ where $ t\in \R $. The fixed points of the involution $\Phi$ lie on that axis $ I $. Then, $ \Phi $ interchanges the fixed points $ \{-1,1\} $. This implies that $\Phi$ will be a reciprocator for $ A $.   
	\begin{cor}
		Let $ f $ in $ \rm PSL(2,\R) $ be a hyperbolic element with fixed point set $ \{-1,1\} $. Then $ f $ is a strongly reciprocal element with reciprocator $ \mu $ such that the fixed points of $ \mu $ are of the form $ e^{it} $ where $ t\in \R $.
	\end{cor}
	The proof follows from the above discussion. 
	\begin{cor} \label{cort}
		If $ A $ is conjugate to $ B $ by $\iota$ in $ \rm PSL(2,\R) $ with $ \theta_{A}\neq 1 $,  then $ \theta_{A}=-\theta_{B} $.
	\end{cor}
	
	\begin{cor} \label{le2}
		Let $A$ be a reciprocal element in the Hecke group $\Gamma_p$ with reversing symmetry $\iota$. Then, $\theta_{A}=0$ or $ 1 $ i.e., every lift of $ A $ in $ \rm SL(2,\R) $ is a symmetric matrix.
	\end{cor} 
	Converse is also true, and that follows from \lemref{hyp}.
	
	Finally, we note the following lemma. 
	
	\begin{lemma}\label{syme}
		Let $ A $ be a hyperbolic element in $ {\rm PSL}(2,\R) $ then, $ \tilde{A}^{k} $ is symmetric if and only if $ \tilde{A} $ is symmetric for any non-zero integer $ k $.
	\end{lemma}
	\begin{proof}
		Let $\tilde{A}=\begin{bmatrix}
			a&&b\\c&&d
		\end{bmatrix}$ be a matrix in ${\rm SL}(2,\R) $, and $ Fix(A)=Fix(A^{k}) $. The fixed points, say  $z_1$ and $z_2$,  of $A$ satisfy this equation,\begin{equation}\label{eq9}
			cz^{2}-(a-d)z-b=0.
		\end{equation}
		If $ \tilde{A} $ is a symmetric matrix, then  $ b=c $ i.e., $ z_{1}z_{2}=-1 $. Since,  $ Fix(A)=Fix(A^{k}) $, and $\tilde{A^{k}}=\begin{bmatrix}
			p&&q\\r&&s
		\end{bmatrix}$. Then $  z_{1},z_{2} $ both are roots of the equation,
		
		\begin{equation}\label{eq98}
			rz^{2}-(p-s)z-q=0.
		\end{equation} 
		Then, $ z_{1}z_{2}=-q/r \imp -1=-q/r\imp q=r.$
		Similarly, the converse is also true.
	\end{proof}
	
	\section{ Proof of \thmref{thm}: For $p$ Odd}\label{odd}
	We note the following observations. 
	\begin{lemma}\label{fp} 
		Let $g\in {\rm PSL}(2, \R)$ be represented by a symmetric matrix $\tilde g\in{\rm SL}(2, \R)$. The map $\iota: z \mapsto -\dfrac{1}{z}$ keeps the fixed point set of $g$ invariant under $ \iota $ . 
	\end{lemma}
	
	\begin{proof} Let $\tilde{g}$, a lift of $ g $ of ${\rm PSL}(2, \R)$ represents  by the matrix $\begin{bmatrix} a  & b \\ c & d \end{bmatrix}$ in ${\rm SL}(2, \R)$. Recall that the fixed point set of $g$ is given by the equation
		$$cz^2 + (d-a) z -b=0.$$
		If $A$ is symmetric, then $b=c$, and hence the above equation is invariant under the map $\iota$.  
	\end{proof}

	\begin{lemma}\label{hyp} 
		Let a hyperbolic element $g$ in $\Gamma_p$ such that $g$ is conjugate to an element in $\Gamma_{p}$ that has a symmetric matrix representation in ${\rm SL}(2, \R)$. Then $ g $ is reciprocal. 
	\end{lemma} 
	\begin{proof} Let $g$ be a hyperbolic element in $\Gamma_p$. Then, without loss of generality, we show that if $\tilde g$ is a symmetric matrix in ${\rm SL}(2, \R)$,  then $g$ is reciprocal with reversing symmetry $\iota$. Also, without loss of generality \ref{syme}, we assume that $g$ is a primitive element in $\Gamma_p$. By \lemref{fp}, the fixed point set of $g$ is invariant under $\iota$. So, $g$ and $\iota g \iota$ have the same set of fixed points. Thus,  $\langle g, \iota g \iota \rangle$ is an elementary two-generator subgroup of $\Gamma_p$. Using arguments as in the proof of \lemref{fuch}, it follows that 
		\[ \iota g \iota= g^{p} \hbox{ for some }  p\in \z-\{0\}. \] 
		But,  $g=\iota g^{p} \iota$, i.e.   $ g=(\iota g \iota)^{p}$.   
		That means $ p $ must be $1$ or $-1$; otherwise, it will be a contradiction to the assumption that $g$ is primitive. If $p=1$,  then  $g$ must lie in the center of $\iota$ in $\Gamma_p$, so $g$ is of order at most two. This is again a contradiction. Hence $p$ must be $-1$, and hence $g$ is reciprocal. 
	\end{proof} 
	\begin{pro}\label{new2}
		The symmetric hyperbolic elements in $ \rm PSL(2,\R) $ are reciprocal, and the reverser is $\iota$. 
	\end{pro}
	\begin{proof}
		The proof follows from the \lemref{hyp} and \lemref{syme}.
	\end{proof}

	\begin{lemma}
		A hyperbolic element $g$ in $\Gamma_p$, is reciprocal if and only if $\tilde g$ is conjugate to an element that has a symmetric matrix representation in ${\rm SL}(2, \R)$. 
	\end{lemma} 
	\begin{proof}
		The necessary condition follows from \lemref{hyp}. We now prove the converse. 
		
		Suppose that $g$ is a reciprocal and primitive element in the Hecke group $\Gamma_p$. Then there is an involution $h$ in $ \Gamma_{p} $ such that $hgh^{-1}=g^{-1}$. Since $p$ is odd, up to conjugacy, $\iota$ is the only involution in $\Gamma_p$. 
		Let $ m\in  \Gamma_p$ be an element such that 
		$ h=m\iota m^{-1} $
		\[h g h^{-1}=g^{-1}\]
		\[ \imp m \iota m^{-1} g m \iota m^{-1} =g^{-1}\]
		\[ \imp \iota (m^{-1}gm)\iota=m^{-1}g^{-1}m\]
		Let $k=m^{-1} g m$. Then $\iota$ conjugates $k$ to $k^{-1}$.
		Since $\iota$ conjugates $k$ to $k^{-1}$,  by \corref{le2},  $  m^{-1}gm $ has symmetric matrix representation. Hence 
		$g$ is conjugate to a symmetric matrix.
	\end{proof}

	\subsection{Proof of \thmref{thm} for odd p } 	Combining \lemref{fuch} and \lemref{pro2}, we obtain the equivalence of (1) and (2). Equivalence between (2) and (3) follows from the above two lemmas. So, the theorem is proved for the odd case. \qed
	
	\section{ Proof of \thmref{thm}: For $p$ Even}\label{even} 
	We have seen reciprocity in $\Gamma_p$ for $p \geq 3$ odd. But for $p$ even,  some questions appear. 
	\begin{enumerate}
		\item Is there any other involution not conjugate to $\iota$?
		\item If it is, then does it work as a reciprocator?
	\end{enumerate}
	The answers to these questions are yes. For the first question, the other involutions are conjugate to $(\iota \alpha_p)^{m}$ where $ 2m=p $. We will denote $\iota \alpha_p$ by $\gamma_{p}$. And for the second one suffices to show the existence of reversing symmetry other than $ \iota $.
	\subsection{Suppose $p$ is even,  $p=2m$,  $m>2$}   The element $ \gamma_{p} $ defined by   $\gamma_{p}: z \mapsto -\dfrac{1}{z+\lambda_p}$, is of order $ p $. Then 
	$\gamma^{m}_{p}$ is an involution in $\Gamma_{p}$,  see \cite{DKS}.
	Lifts of $\gamma_{p}$ in ${\rm SL}(2,\R)$ are 
	$ \pm \begin{bmatrix}
		0 && -1\\
		1 && {\lambda_p} 
	\end{bmatrix} $. To determine the involution $\gamma^{m}_{p}$, we need the following lemma.
	\begin{lemma}
		Let $A=  \begin{bmatrix}
			0 && -1\\
			1 && \lambda_p 
		\end{bmatrix}  $ be the lift of $\gamma_{p}$ in ${\rm SL}(2,\R) $ . Then, for any natural number $ k $, $ A^{k}=\begin{bmatrix}
			-a_{k-1}&& -a_{k}\\
			a_{k}&& a_{k+1}
		\end{bmatrix} $ where $ a_{k+1}=-a_{k-1}+\lambda_p a_{k} $ and $ a_{0}=0,a_{1}=1 $ . 
		
	\end{lemma}
	\begin{proof}
		We will use induction on $k$. For $k=1$, the matrix $ A $ is $ \begin{bmatrix}
			-a_{0}&& -a_{1}\\
			a_{1}&& a_{2}
		\end{bmatrix} $, i.e.,  $ \begin{bmatrix}
			0 && -1\\
			1 && \lambda_p 
		\end{bmatrix}. $ So the lemma is true for $k=1$.\\
		Let us consider the lemma is true for $k=n$.
		So, that means $ A^{n}=   \begin{bmatrix}
			0 && -1\\
			1 && \lambda_p\end{bmatrix}^{n}$
		\[ \imp A^{n}=\begin{bmatrix}
			-a_{n-1}&& -a_{n}\\
			a_{n}&& a_{n+1}
		\end{bmatrix} , \] where $ a_{n+1}=-a_{n-1}+\lambda_p $. 
		Then \[ A^{n+1}=\begin{bmatrix}
			-a_{n-1}&& -a_{n}\\
			a_{n}&& a_{n+1}
		\end{bmatrix}\begin{bmatrix}
			0 && -1\\
			1&& \lambda_p
		\end{bmatrix}\]
		\[ = \begin{bmatrix}
			-a_{n}&& -a_{n+1}\\
			a_{n+1}&& -a_{n}+ \lambda_p a_{n+1}
		\end{bmatrix}\]
		\[= \begin{bmatrix}
			-a_{n}&& -a_{n+1}\\
			a_{n+1}&& a_{n+2}
		\end{bmatrix}  ,\] since,   $ a_{k+1}=-a_{k-1}+\lambda_p a_{k}$. 
		So the lemma is true for $k=n+1$. Hence the lemma is proved.
	\end{proof}
	\begin{lemma}\label{B}
		The element   $\gamma_{p}^m$ in $\Gamma_q$  lifts to the following matrices in ${\rm SL}(2, \R)$: 
		
		$$\pm \begin{bmatrix}
			-\frac{\cos\frac{\pi}{p}}{\sin\frac{\pi}{p}}&& -\frac{1}{\sin\frac{\pi}{p}}\\
			\frac{1}{\sin\frac{\pi}{p}}&& \frac{\cos\frac{\pi}{p}}{\sin\frac{\pi}{p}}	
		\end{bmatrix}.$$
	\end{lemma}
	
	\begin{proof}
		Let us consider $ p=2m $. Let $ A $ be one of the lifts of $\gamma_{p}$ in $ \rm SL(2,\R) $. So, $ B=A^{m}=\begin{bmatrix}
			-a_{m-1}&& -a_{m}\\
			a_{m}&& a_{m+1}
		\end{bmatrix}$ where $ a_{m+1}=-a_{m-1}+\lambda_p $, is a lift of the involution $\gamma_{p}^{m}$. Thus, then $ B^{2}=\pm Id
		$ 
		\[ \begin{bmatrix}
			-a_{m-1}&& -a_{m}\\
			a_{m}&& a_{m+1}
		\end{bmatrix}^{2}=\pm Id\]
		\[\imp \begin{bmatrix}
			a^{2}_{m-1}-a^{2}_{m}&& a_{m-1}a_{m}-a_{m}a_{m+1}\\
			a_{m-1}a_{m}-a_{m}a_{m+1}&& 	a^{2}_{m+1}-a^{2}_{m}
		\end{bmatrix}=\pm Id \]
		\[ \imp 	a^{2}_{m-1}-a^{2}_{m}=\pm1~~~~~ \\ \hbox{ , } \\ ~ ~~~~  a_{m-1}a_{m}-a_{m}a_{m+1}=0 ~~~~~ \\ \hbox{ , } \\ ~ ~~~~ a^{2}_{m+1}-a^{2}_{m}=\pm1.\]
		From the second equation, if we obtain two cases.\n
		\textbf{Case 1. } 
		$ a_{m}=0 $, then from the first and third equation $ a^{2}_{m+1}=1 \imp a_{m+1}=\pm 1.~~~~~~~~\hbox{Similarly } a_{m-1}\ =\pm 1.$ Also for the determinant condition $ a_{m-1} $ and $  a_{m+1} $ have opposite signs.  So $ B  $ must projects to the identity of $ {\rm PSL}(2,\R) $ . 
		This contradicts our hypothesis.\n
		\textbf{Case 2. }
		$ a_{m}\neq 0 $, then $ a_{m+1}=a_{m-1} $ then 
		\[a_{m+1}=-a_{m-1}+\lambda_p  a_{m} \]
		\[\imp  2a_{m+1}=\lambda_p a_{m}\]
		\[\imp a_{ m+1 }=\cos\frac{\pi}{p}a_{m}. \]
		Then $ a^{2}_{ m+1 }\leq a^{2}_{ m } $ and 
		\[a^{2}_{m+1}-a^{2}_{m}=-1 \]
		\[ \imp a^{2}_{m}-\cos^{2}\frac{\pi}{p}a^{2}_{m}=0 \]
		\[\imp a_{m}=\pm \frac{1}{\sin\frac{\pi}{p}}.\]
		So the lifts of  $ \gamma_{p}^{m}  $ are $\pm \begin{bmatrix}
			-\frac{\cos\frac{\pi}{p}}{\sin\frac{\pi}{p}}&& -\frac{1}{\sin\frac{\pi}{p}}\\
			\frac{1}{\sin\frac{\pi}{p}}&& \frac{\cos\frac{\pi}{p}}{\sin\frac{\pi}{p}}	
		\end{bmatrix} $ in ${\rm SL}(2,\R) $ .
	\end{proof}\n
	Now, we show the existence of reversing symmetry conjugate to $\gamma^{m}_{p}$.
	\begin{lemma}\label{ext}
		Any reciprocal hyperbolic element in $ \Gamma_{p} $ have reciprocators conjugate to either $ \iota $ or $ \gamma^{m}_{p} $. 
	\end{lemma}
	\begin{proof}
		Let us consider $ M $ to be a reciprocal element in the Hecke group $ \Gamma_{p} $. So, by \cite[Lemma 3.1]{HR}, $ M $ is conjugate to \[W= V_{j_{1}}V_{j_{2}}\dots V_{j_{n}},\]
		
		where $ V_{j_{k}}=( U)^{j_{k}-1}\ap $, $ U= \ap\iota $, $ 1\leq j_{k}\leq p-1 $ for $ 1\leq k\leq n $ and $ n\in \mathbb{N}  $. The product is unique up to cyclic permutation. Since reciprocity is invariant under the conjugation, by \cite[Theorem 1.4]{MKD}, the reduced form of $ W $ and $ W^{-1} $ differ by a cyclic permutation. Then,
		\[ W^{-1}=V^{-1}_{j_{n}}V^{-1}_{j_{n-1}}\dots V^{-1}_{j_{1}}\]
		\[\imp W^{-1}=\ap^{-1}U^{p-j_{n}+1}\ap^{-1}U^{p-j_{n-1}+1}\dots \ap^{-1}U^{p-j_{1}+1}\] 
		\[\imp W^{-1}=\ap^{-1}(\ap \iota\dots \ap \iota) \ap^{-1}(\ap\iota\ap\iota\dots \ap\iota)\dots \ap^{-1}(\ap\iota\ap\iota\dots \ap\iota)\]
		\[ \imp W^{-1}= \iota (\ap\iota\ap\iota\dots \ap\iota) \iota (\ap\iota\ap\iota\dots \ap\iota) \iota\dots \iota (\ap\iota\ap\iota\dots  \ap\iota)\]
		\[\imp W^{-1}=\iota(\ap\iota)^{p-j_{n}-1}\ap(\ap\iota)^{p-j_{n-1}-1}\ap\dots (\ap\iota)^{p-j_{1}-1}\ap\iota\]
		\[ \imp W^{-1}= \iota V_{p-j_{n}}V_{p-j_{n-1}}\dots V_{p-j_{1}}\iota \]
		\[\imp \iota W^{-1}\iota=V_{p-j_{n}}V_{p-j_{n-1}}\dots V_{p-j_{1}}=P~(say).\]
		$ P $ is a cyclically reduced word in $ \Gamma_{p} $. Then it is a cyclic permutation of W. That implies, $ RPR^{-1}=W $ where $ R= (V_{j_{1}}V_{j_{2}}\dots V_{j_{r}})^{\pm1}$ for $ r\leq n $ and let $ q=n-r $. Then either $ R $ or $ R^{-1} $ be a sub-word of $ W $. Without loss of generality, let us assume $ R $ is a sub-word. Then,
		$ RPR^{-1}=W\imp R\iota W^{-1}\iota R^{-1}=W  $. By \lemref{fuch} it follows that  $ R\iota$ is an involution. 
		Then \[ R\iota=  V_{j_{1}}V_{j_{2}}\dots V_{j_{r}}\iota\]\[
		=U^{j_{1}-1}\alpha_{p}U^{j_{2}-1}\alpha_{p}\dots U^{j_{r}-1}\alpha_{p}\iota\]	
		\[ =U^{j_{1}-1}\alpha_{p}\iota \iota U^{j_{2}-1}\alpha_{p}\iota \iota\dots U^{j_{r}-1}\alpha_{p}\iota\]
		\[ = U^{j_{1}} \iota  U^{j_{2}} \iota \dots  U^{j_{r}}.\]
		Therefore,
		\[ R \iota R \iota =Id\]
		\[\imp U^{j_{1}} \iota  U^{j_{2}} \iota \dots  U^{j_{r}}U^{j_{1}} \iota  U^{j_{2}} \iota \dots  U^{j_{r}}=Id \]
		\[\imp U^{j_{1}} \iota  U^{j_{2}} \iota \dots  U^{j_{r}+j_{1}} \iota  U^{j_{2}} \iota \dots  U^{j_{r}}=Id.\]
		Since every $ j_{k} \geq 1$, then the above equation will hold if the following conditions are satisfied:
		$ j_{r}+j_{1}=p, j_{r-1}+j_{2}=p,\dots , j_{r-i}+j_{1+i}=p,..j_{1}+j_{r}=p$. Therefore,
		\[ P=R^{-1}WR\]
		\[\imp V_{p-j_{n}}V_{p-j_{n-1}}\dots V_{p-j_{1}}=V_{j_{r+1}}V_{j_{r+2}}\dots V_{j_{r+q}}V_{j_{1}}V_{j_{2}}\dots V_{j_{r}}\]
		\[\imp U^{p-j_{n}-1}\alpha_{p}U^{p-j_{n-1}-1}\alpha_{p}\dots U^{p-j_{1}-1}\alpha_{p}=U^{j_{r+1}-1}\alpha_{p}U^{j_{r+2}-1}\alpha_{p}\dots U^{j_{r+q}-1}\alpha_{p}U^{j_{1}-1}\alpha_{p}\dots U^{j_{r}-1}\alpha_{p}\] 
		\[\imp U^{p-j_{n}-1}\alpha_{p}\iota \iota  U^{p-j_{n-1}-1}\alpha_{p} \iota \iota \dots U^{p-j_{1}-1}\alpha_{p} \iota \iota\] \[ =U^{j_{r+1}-1}\alpha_{p} \iota \iota U^{j_{r+2}-1}\alpha_{p} \iota \iota \dots U^{j_{r+q}-1}\alpha_{p} \iota \iota U^{j_{1}-1}\alpha_{p} \iota \iota \dots U^{j_{r}-1}\alpha_{p} \iota \iota  \]
		\[\imp U^{p-j_{n}}\iota U^{p-j_{n-1}}\iota \dots U^{p-j_{1}}\iota= U^{j_{r+1}}\iota U^{j_{r+2}}\iota \dots U^{j_{r+q}}\iota U^{j_{1}}\iota \dots U^{j_{r}}\iota.  \]
		
		This implies, $ j_{n}+j_{r+1}=p , j_{n-1}+j_{r+2}=p,\dots ,j_{n-k+1}+j_{r+k}=p,\dots ,j_{n-q+1}+j_{r+q}=p$. Let,
		$\sigma(x)$ denotes a cyclicly reduced form of $ x $, i.e., $ \sigma(x^{-1}yx)=\sigma(y) $ for $ x,y  $ are elements of $ \Gamma_{p} $.

		Then the cyclic reduced form of $ R \iota  $ i.e. $ \sigma(R \iota )=\sigma(U^{j_{1}} \iota  U^{j_{2}} \iota \dots  U^{j_{r}}) $, it follows that the reduced form of $ R \iota  $ is either  $ \iota $  or $ U^{m} $ where $ 2m=p $. Also $ U^{m} $ is conjugate to $ \gamma^{m}_{p} $.
		That means $ R \iota  $ is conjugate to $  \iota  $ or $ \gamma_{p}^{m} $ depending on $ r $ is even or odd, respectively.
	\end{proof}
	\medspace\\
	
	\textbf{Note:} In the proof of the \lemref{ext}, we just show that $\gamma_{p}^{m}$ will exist as a reversing symmetry for some reciprocal elements. But, eventually, it provides more information. Let us consider $ W=RQ $ where $ R,Q $ are sub-words of $ W $, and we know that $ R\iota $ is the involution and reversing symmetry of $ W $. Then, $ \iota R^{-1}=R\iota .$ Also, $ R\iota W $ will be an involution. Then,
	\[ W=RQ=R\iota\iota Q\]
	\[\imp R\iota W=(R\iota)^{2}\iota Q=\iota Q.\]
	That means $ \iota Q $ is an involution. Then, $ \iota Q $ is also a reversing symmetry of $ W $.
	\[ Q=V_{j_{r+1}}V_{j_{r+2}}\dots V_{j_{r+q}},\]
	where, $ r+q=n $. Then this implies,
	\[ \iota Q= \iota U^{j_{r+1}-1}\alpha_{p}U^{j_{r+2}-1}\alpha_{p}\dots U^{j_{r+q}-1}\alpha_{p}\]
	\[=\iota U^{j_{r+1}-1}\alpha_{p}\iota \iota U^{j_{r+2}-1}\alpha_{p}\iota \iota\dots U^{j_{r+q}-1}\alpha_{p}\iota \iota\]
	\[=\iota U^{j_{r+1}} \iota U^{j_{r+2}} \iota\dots U^{j_{r+q}}\iota\]
	\[\sigma(\iota Q)=\sigma(\iota U^{j_{r+1}} \iota U^{j_{r+2}} \iota\dots U^{j_{r+q}}\iota).\]
	If $ q $ is even, then, $ \sigma(\iota Q)=\sigma(\iota) $. If $ q $ is odd, then $ \sigma (\iota Q) =\sigma (\gamma_{p}^{m})$.
	So, the conclusion is that the reversing symmetry of $ W $ depends on the value of $ n $. If $ n  $ is odd, then the reciprocators of reciprocal element conjugate to $ \iota  $ and $ \gamma_{p}^{m} $ respectively. If n is even, then the reciprocator conjugate either to $\iota$ (if $ r$ and $ q $ are both even ) or $\gamma^{m}_{p}$(if $ r $ and $ q $ are both odd).
	For example, if we choose \[ W=V_{1}V_{p-1}V_{m}\]
	where, $ V_{j}=U^{j-1}\ap. $ \[W= \ap (\ap \iota)^{p-2}\ap  (\ap \iota)^{m-1}\ap\]
	\[=\ap \iota \iota (\ap \iota)^{p-2}\ap \iota\iota(\ap \iota)^{m-1}\ap \iota \iota \]
	\[=(\ap \iota)\iota  (\ap \iota)^{p-1}\iota (\ap \iota)^{m}\iota\]
	\[=\iota(\iota \ap )\iota (\iota \ap)^{p-1}\iota(\iota\ap)^{m}\]	
	\[=\iota(\gamma_{p})\iota \gamma_{p}^{p-1}\iota \g. \]
	So, one of the lifts of $ W $ is, $ \dfrac{1}{\sqrt{1-(\frac{\lambda_{p}}{2})^{2}}}\begin{bmatrix}
		(3\lambda_{p}^2+2)/2&& (\lambda_p^3+3\lambda_{p})/2\\
		3\lambda_p/2&&(\lambda_{p}^2+2)/2
	\end{bmatrix} .$ We notice that $ \theta_{W}=\lambda_{p}/2=\cos\pi/p. $ If we take the conjugate element of $ W $, $ \gamma_{p}^{p-1}\iota W \iota \gamma_{p}$, then one of the lifts of this element is, $ \dfrac{1}{\sqrt{1-(\frac{\lambda_p}{2})^{2}}}\begin{bmatrix}
		1 && \frac{3\lambda_p}{2}\\ \frac{3\lambda_p}{2}&& 2\lambda_p^{2}+1
	\end{bmatrix} $, which is a symmetric matrix with fixed point ratio $ 0 $.
	
	\begin{cor}\label{co3}
		{Some infinite order reciprocal elements have two types of reciprocators one is conjugate to $\iota$, and the other one is conjugate to $\gamma_{p}^{m}$. }
	\end{cor}
	Proof of the \corref{co3} is followed by the previous note.
	\begin{pro}
		Let $ A $ be a hyperbolic element in $ \Gamma_{p} $ with $ \theta_{A}=\cos\dfrac{\pi}{p} $. Then $ A $ is a reciprocal element in $ \Gamma_{p} $ with reciprocator $ \gamma^{m}_{p} $.
	\end{pro}
	\begin{proof}
		The proof of the proposition follows from the \lemref{co1}.
	\end{proof}
	\begin{lemma}\label{le1}
		Let $ g $ be a hyperbolic reciprocal element in $\Gamma_{p}$ with reversing symmetry $ h $, conjugate to $\gamma^m_p$. Then,  it is conjugate to an element $ A $ in $\Gamma_{p}$ such that $ \theta_{A}=\cos\dfrac{\pi}{p} $ or,  $ \theta_{A}=1 $.
	\end{lemma}\label{lem8}
	\begin{proof}
		$  g $ is a hyperbolic reciprocal element in $ \Gamma_{p} $  and \[ hgh^{-1}=g^{-1}.\]
		So,$  h $ must be conjugate to $ \iota $ or $ \gamma_{p}^{m} $. Suppose, $h$ conjugate to $\gamma_{p}^{m}$ then, there exists 
		$  t\in \Gamma_{p}$ s.t.
		\[ tht^{-1}=\gamma_{p}^{m},\]
		then,\[tht^{-1}tgt^{-1}tht^{-1}=tg^{-1}t^{-1}.\]
		Let us consider $ tgt^{-1} $ as a matrix $ \begin{bmatrix}
			a&& b\\
			c&& d
		\end{bmatrix} $ in $ \rm SL(2,\R) $ and $ B $ is one of the lifts of $ \gamma_{p}^{m} $ from \lemref{B}. Then
		\[ B \begin{bmatrix}
			a&& b\\
			c&& d
		\end{bmatrix}B^{-1}= \pm\begin{bmatrix}
			a&& b\\
			c&& d
		\end{bmatrix}^{-1}\]
		\[\imp B \begin{bmatrix}
			a&& b \\
			c&&d
		\end{bmatrix}=\pm \begin{bmatrix}
			a&&b \\
			c&&d
		\end{bmatrix}^{-1}B\]
		\[\imp \pm \frac{1}{\sin\frac{\pi}{p}} \begin{bmatrix}
			-\cos\frac{\pi}{p}&& -1\\
			1&& \cos\frac{\pi}{p}
		\end{bmatrix} \begin{bmatrix}
			a&& b\\
			c&& d
		\end{bmatrix} =\mp \dfrac{1}{\sin\frac{\pi}{p}}\begin{bmatrix}
			d&& -b\\
			-c&& a
		\end{bmatrix}\begin{bmatrix}
			-\cos\frac{\pi}{p}&& -1\\
			1&& cos\frac{\pi}{p}
		\end{bmatrix}\]
		\[\imp \begin{bmatrix}
			-\frac{\lambda_p}{2}&&-1\\1&& \frac{\lambda_{p}}{2}
		\end{bmatrix} \begin{bmatrix}
			a&& b\\
			c&& d
		\end{bmatrix} =\mp \begin{bmatrix}
			d&& -b\\
			-c&& a
		\end{bmatrix} \begin{bmatrix}
			-\frac{\lambda_{p}}{2}&&-1\\1&& \frac{\lambda_{p}}{2}
		\end{bmatrix} ~~~~~\hbox{   where, $ \lambda_{p} =2\cos\frac{\pi}{p}$}\]
		\[ \imp \begin{bmatrix}
			-a\frac{\lambda_{p}}{2}-c&& -b\frac{\lambda_{p}}{2}-d\\ a+\frac{\lambda_{p}}{2}c&& b+d\frac{\lambda_{p}}{2}
		\end{bmatrix}=\mp \begin{bmatrix}
			-\frac{\lambda_{p}}{2}d-b&& -d-b\frac{\lambda_{p}}{2}\\\frac{\lambda_{p}}{2}c+a&& c+a\frac{\lambda_{p}}{2}
		\end{bmatrix} .\]
		\textbf{Case 1.}
		\[  \begin{bmatrix}
			-a\frac{\lambda_{p}}{2}-c&& -b\frac{\lambda_{p}}{2}-d\\ a+\frac{\lambda_{p}}{2}c&& b+d\frac{\lambda_{p}}{2}
		\end{bmatrix}=- \begin{bmatrix}
			-\frac{\lambda_{p}}{2}d-b&& -d-b\frac{\lambda_{p}}{2}\\\frac{\lambda_{p}}{2}c+a&& c+a\frac{\lambda_{p}}{2}
		\end{bmatrix} . \] 
		Then, \[ -b\frac{\lambda_{p}}{2}-d=0 \hbox{  and  }  a+\frac{\lambda_{p}}{2}c=0 \]
		\[\imp \frac{\lambda_{p}}{2}(b+c)+a+d=0.\]
		Here, $ b+c\neq 0 $ since, $ a+d>2 $.
		Also, \[ -a\frac{\lambda_{p}}{2}-c=\frac{\lambda_{p}}{2}d+b.\]
		From the above equations we obtain that \[\frac{\lambda_{p}}{2}=-\frac{b+c}{a+d} \hbox{  and } \frac{\lambda_{p}}{2}=-\frac{a+d}{b+c}.\] Then, $\frac{\lambda_{p}^2}{4}=1$ which is contradiction. So, this case will not appear.

		\textbf{Case 2. }
		\[  \begin{bmatrix}
			-a\frac{\lambda_{p}}{2}-c&& -b\frac{\lambda_{p}}{2}-d\\ a+\frac{\lambda_{p}}{2}c&& b+d\frac{\lambda_{p}}{2}
		\end{bmatrix}= \begin{bmatrix}
			-\frac{\lambda_{p}}{2}d-b&& -d-b\frac{\lambda_{p}}{2}\\\frac{\lambda_{p}}{2}c+a&& c+a\frac{\lambda_{p}}{2}
		\end{bmatrix} . \] This implies,
		\[ -a\frac{\lambda_{p}}{2}-c=-\frac{\lambda_{p}}{2}d-b\]
		\[\imp (a-d)\frac{\lambda_{p}}{2}=b-c.\]
		\textit{Sub case 1.} If the lifts are symmetric matrices i.e., $ b=c $ then, $ a=d $. This and \propref{new2} implies that if the lifts have the same diagonal and off-diagonal entries, then they have reciprocators as $\iota$ and $\gamma_{p}^{m}$ both and $\theta_{A}=1$.\\
		\textit{Sub case 2.} If the lifts are non-symmetric matrices then \[\frac{\lambda_{p}}{2}=\dfrac{b-c}{a-d}.\]
		
		So, the reciprocal elements whose reversing symmetry is  $\gamma^{m}_{p}$, conjugate to a matrix $A=\begin{bmatrix}
			a&& b\\
			c&& d
		\end{bmatrix} $ such that either  $ \theta_{A}=\cos\dfrac{\pi}{p} $ or,  $ b=c $ and $ a=d $. 
	\end{proof}
	\subsection{Proof of \thmref{thm} for even $p$}
	Combining \lemref{fuch}, \lemref{pro2} and \lemref{ext}, we obtain that the reciprocal elements of $ \Gamma_{p} $ are either involutions or hyperbolic elements with reciprocator $\iota$ or $ \gamma^{m}_{p} $. \lemref{le1} implies for reciprocator $ \gamma^{m}_{p} $, the reciprocal element is conjugate to an element whose fixed point ratio is $ \cos\dfrac{\pi}{p} $ and \corref{le2} implies for reciprocator $ \iota$, the reciprocal element is conjugate to an element whose fixed point ratio is $ 0 $. Hence, the theorem is proved. \qed

	\section{parametrization of the Reciprocal classes in $\Gamma_p$ }\label{param} 
	
	{ \subsection{Proof of \thmref{sym}}
		Without loss of generality,  consider a primitive reciprocal hyperbolic element $ r $ in $\Gamma_{p}$. 
		Let $ S $ be a  reversing symmetry for $ r $. Then,
		\begin{equation}\label{1}
			S^{-1}rS=r^{-1}.
		\end{equation}
		This implies,
		\begin{equation}\label{k}
			r^{k}Sr^{k}=S ~~~~~~~~\hbox{$ \forall k\in \z $.}
		\end{equation}
		Then, by \lemref{fuch}, $ S $ conjugate to either $ \iota $ or $\gamma_{p}^{m}$.
		
		First, suppose $ S $ is conjugate to $\iota$. Then, let $ \eta\in \Gamma_{p} $ be an element such that,
		\begin{equation}\label{2}
			\eta^{-1} S\eta=\iota.
	\end{equation}}
	There are two solutions for the equation \eqref{2}. If one is $\eta_{S} $, then another one is $ S \eta_{S} $. This gives two hyperbolic symmetric elements, $ \eta_{S}^{-1}r\eta_{S} $ and inverse of it.\\
	Let us consider another symmetric element in the conjugacy class of $ r $, say $ \beta^{-1} r \beta  $.
	Since it is symmetric, then 
	\begin{equation}\label{3}
		\iota \beta^{-1} r \beta \iota =\beta^{-1} r^{-1} \beta.
	\end{equation}
	That means, \[ \beta\iota \beta^{-1}r \beta \iota \beta^{-1}=r^{-1}\]
	\[ \imp (\beta\iota \beta^{-1})r (\beta \iota \beta^{-1})=r^{-1}\]
	\[ \imp \beta\iota \beta^{-1}r \beta \iota \beta^{-1}=S^{-1}rS\]
	\[\imp S\beta\iota \beta^{-1}r \beta \iota \beta^{-1}S^{-1}=r\]
	\[ \imp (\beta\iota \beta^{-1}S^{-1})^{-1}r \beta \iota \beta^{-1}S^{-1}=r.\]
	That implies, $ \beta\iota \beta^{-1}S^{-1} \in C(r) $ Centralizer of r. Then,
	$ \beta \iota \beta ^{-1}=bS $, where $ b\in C(r) $ i.e., \begin{equation}\label{n}
		b=r^{n} \\ \\~~~~~~~~~~ \hbox{for some $ n\in \z $.}\\
	\end{equation} 
	Suppose  $ \ps=r^{2k}S $ for $ k\in \z $. This implies \begin{equation}\label{ke}
		r^{-k}\ps r^{k}=r^{k}Sr^{k}.
	\end{equation}   Combining \eqref{ke} and \eqref{k}, we obtain, \begin{equation}\label{ne}
		r^{-k}\ps r^{k}=S.
	\end{equation} Again,
	$ \ps  $ also satisfy \eqref{1} and \eqref{2}. So, $ \eta_{\ps} $ and $ \ps \eta_{\ps} $ are the solutions for \begin{equation}
		\eta^{-1}\ps \eta=\iota.
	\end{equation}
	Then, $ (\eta_{\ps}\eta_{S}^{-1})^{-1}r^{k}\in C(S)=\{Id,S\} $. If \[  (\eta_{\ps}\eta_{S}^{-1})^{-1}r^{k}=Id.\] Then it implies, \[ (\eta_{\ps}\eta_{S}^{-1})=r^{k}\] \[\imp \eta_{\ps}=r^{k}\eta_{S}\] \[\imp \eta_{\ps}^{-1}r\eta_{\ps}=\eta_{S}^{-1}r\eta_{S} .\]
	If $ \eta_{\ps}=r^{k}S\eta_{S} $. This implies, \[\eta_{\ps}=r^{k}S\eta_{S}\]
	\[\imp \eta_{\ps}^{-1}r\eta_{\ps}=(r^{k}S\eta_{S})^{-1}rr^{k}S\eta_{S} \]
	\[ = \eta_{S}^{-1}S^{-1}rS\eta_{S}\]
	\[=\eta_{S}^{-1}r^{-1} \eta_{S}.\] 
	Thus, if two reciprocators differ by an even factor of $ r $, then the pair of symmetric elements related to the respective reciprocators are the same.\\
	Similarly, suppose the reciprocator of $ r  $, $ S $ conjugate to $ \gamma_{p}^{m} $. Then, 
	\begin{equation}\label{10}
		S^{-1}rS=r^{-1}.
	\end{equation}
	And there is an element $ \eta \in \Gamma_{p}  $ such that,
	\begin{equation}\label{11}
		\eta^{-1} S\eta=\gamma_{p}^{m}
	\end{equation}
	\[\imp S=\eta \gamma_{p}^{m}\eta^{-1}\]
	\[\imp S=(\eta \gamma_{p}\eta^{-1})^{m}\]
	So, $ C(S)\simeq \langle \eta \gamma_{p}\eta^{-1}\rangle $ in $\Gamma_{p}$, then let $ \alpha $ be the generator of  $ C(S) $, such that $ S=\alpha^{m} $. Let us consider $ \eta_{S} $ and $ \eta_{S}^{\prime} $ are two solutions for \eqref{11}.
	\[\eta_{S}^{-1}S\eta_{S}=\gamma_{p}^{m}=(\eta_{S}^{\prime})^{-1}S\eta_{S}^{\prime}. \]
	This implies, \[ \eta_{S}(\eta_{S}^{\prime})^{-1}\in  \langle \alpha \rangle\]
	\[\imp \eta_{S}=\alpha^{i}(\eta_{S}^{\prime})\]
	where, $ i=0,1,\dots, p-1 $.
	Then, there are $ p $ solutions for the equation \eqref{11} for each $ S $. Those are in the set $ \mathcal{P}=\{\eta_{S},\alpha\eta_{S},\dots , \alpha^{p-1}\eta_{S}\} $ and they will generate $ p $-reciprocal elements and their inverses in the same reciprocal class.\\ 
	
	Let $ \eta_{S}^{\prime} \in \mathcal{P}$ i.e., $ \eta_{S}^{\prime}=\alpha^{i} \eta_{S} $. So, \[(\eta_{S}^{\prime})^{-1}r\eta_{S}^{\prime}=\eta_{S}^{-1}\alpha^{-i}r\alpha^{i}\eta_{S}
	\]
	\[=\eta_{S}^{-1}\alpha^{m-i}S^{-1}rS\alpha^{-m+i}\eta_{S}\]
	\[=\eta_{S}^{-1}\alpha^{m-i}r^{-1}\alpha^{-m+i}\eta_{S}\]
	\[=(\eta_{S}^{-1}\alpha^{m-i}r\alpha^{-m+i}\eta_{S})^{-1}.\]
	That means, for the reciprocator $ S $, there are $ p $ $ p $-reciprocal elements. If we choose another $ p $-reciprocal element distinct from the previous ones, i.e., $ \beta r \beta^{-1}$ in the reciprocal class of $ r $. Then, from the previous equation, \eqref{3}, we obtain similar result i.e.,  $ \beta \gamma_{p}^{m}\beta^{-1} \in C(r)$. Then, $ \beta \gamma_{p}^{m} \beta^{-1}= bS $ for some $ b\in C(r) $ where  $ b=r^{n} $ for some $ n\in \z $. Similar to the previous case, we  suppose that $ \beta \gamma_{p}^{m} \beta^{-1}= \ps $ and $ n=2k $. Then, by \eqref{ke} and \eqref{ne}, $ \ps $ satisfies the equation, \begin{equation}\label{4}
		\eta^{-1}\ps \eta=\gamma_{p}^{m}
	\end{equation} for some $ \eta \in \Gamma_{p}$. Let the solutions for each $ \ps $ be $ \eta_{\ps} $ ($\beta$ is one of the solutions for the above equation). Then,
	\[ \eta_{\ps}^{-1}r^{k}Sr^{-k}\eta_{\ps}=\eta_{S}^{-1}S\eta_{S}.\]
	This implies, $ \eta_{S}\eta_{\ps}^{-1}r^{k}\in C(S)=\langle \alpha\rangle  $. So,
	\begin{equation}\label{30}
		\eta_{S}\eta_{\ps}^{-1}r^{k}=\alpha^{i}
	\end{equation} for some $ 0\leq i\leq p-1. $ Therefore,\[r^{k}=\eta_{\ps}\eta_{S}^{-1}\alpha^{i}\]
	\[\imp r^{k}\alpha^{p-i}\eta_{S}=\eta_{\ps}\]
	\[\imp \eta_{\ps}^{-1}r \eta_{\ps}=(\alpha^{p-i}\eta_{S})^{-1}r\alpha^{p-i}\eta_{S}.\] 
	That means we get the same set of $ p $-reciprocal elements if $ \ps=r^{2k}S $.

	It remains to show the odd $ n $ case. First, we claim that if $ \ps=r^{2k+1}S $ and $ \ps $, $ S $ both are conjugate to the same involution,  then, $ \eta_{S}^{-1}r\eta_{S} $ and $\eta_{\ps}^{-1}r\eta_{\ps} $ are distinct for any $ \eta_{S} $, $ \eta_{\ps} $ satisfying \eqref{2}, \eqref{11} respectively.
	
	If possible, suppose,
	\[\eta_{S}^{-1}r\eta_{S}=\eta_{\ps}^{-1}r\eta_{\ps}\] 
	\[\imp \eta_{S}\eta_{\ps}^{-1}r\eta_{\ps}\eta_{S}^{-1}=r.\]
	Then, $ \eta_{\ps}=b\eta_{S} $ where $ b\in C(r) $. Therefore,
	\begin{equation}\label{key}
		\eta_{S}^{-1}S\eta_{S}=\eta_{\ps}^{-1}\ps \eta_{\ps}
	\end{equation}
	\[\imp \eta_{\ps}\eta_{S}^{-1}S\eta_{S}\eta_{\ps}^{-1}=\ps\]
	\[\imp bSb^{-1}=\ps.\] $ S $ is the involution then, $ r^{t}S $  is also an involution for any $ t\in \z $. That means, \[ Sb^{-1}=\ps b\]
	\[\imp Sb^{-1}=\ps b\]
	\[\imp Sb^{-2}=\ps.\]
	This implies,
	\begin{equation}\label{m}
		\ps =b^{2}S
	\end{equation}where, $ b\in C(r) $. This contradicts our assumption. \\ The same claim will not work for the case where the reciprocators $ S $ and $ \ps $ are the solutions of \eqref{2} and \eqref{4}, respectively. To show that suppose \begin{equation}\label{e4}
		\eta_{S}^{-1}r\eta_{S}=\eta_{\ps}^{-1}r\eta_{\ps}.
	\end{equation}
	That implies,
	\[ \eta_{S}^{-1}S\eta_{S}\eta_{S}^{-1}S^{-1}rS\eta_{S}\eta_{S}^{-1}S\eta_{S}=\eta_{\ps}^{-1}\ps\eta_{\ps}\eta_{\ps}^{-1}\ps^{-1}r\ps \eta_{\ps}\eta_{\ps}^{-1}\ps\eta_{\ps} \]
	\[\imp \iota \eta_{S}^{-1} r^{-1}\eta_{S}\iota=\g \eta_{\ps}^{-1}r^{-1}\eta_{\ps} \g \] \[\imp \g \iota (\eta_{S}^{-1}r^{-1}\eta_{S})\iota\g=(\eta_{\ps}^{-1}r^{-1}\eta_{\ps})\]
	\[\imp \g \iota (\eta_{S}^{-1}r^{-1}\eta_{S})\iota\g =(\eta_{S}^{-1}r^{-1}\eta_{S})\]
	\[ \imp (\eta_{S}^{-1}r\eta_{S})=(\iota\g)^{k}~~~~~~~~\hbox{   where, $ k\in \z .$ }\]
	Since $ r $ is a primitive element, $ k $ is either $ 1 $ or $ -1 $.
	So, the elements conjugate to any non-zero power of $ \iota \g $ have exactly $ p $ \textit{p-reciprocal} elements and exactly two  \textit{symmetric p-reciprocal } elements. If there is any other \textit{symmetric p-reciprocal } element, it must be a power of $ r $ (since r is a primitive element.). This leads to a contradiction. If possible, let there is another $\eta$ such that it satisfies either \eqref{2} or \eqref{4} with \eqref{e4}. Without loss of generality, we choose $\eta$ satisfy \eqref{4} then,\[ \eta_{S}^{-1}\alpha^{-i}r\alpha^{i} \eta_{S}=\eta_{S}^{-1}r\eta_{S}.\]
	This implies, \[\alpha^{-i}r\alpha^{i}=r\] which is a contradiction for any non-zero $ i $. So there are only two elements $\eta_{S}$ and $\eta_{\ps}$ such that $ \eta_{S}^{-1}r\eta_{S}=\eta_{\ps}^{-1}r\eta_{\ps} $.
	
	Now fix a reciprocator $ S $. We can divide the class of reciprocators in terms of $ [S] $ and $ [rS] $ as reciprocators of each class will generate the same set of symmetric elements or $ p $-reciprocal elements. If $ S $ and $ rS$ are both conjugate to either $\iota$ or $\gamma_{p}^{m}$,  the proof of part $(a) $ and $ (b) $ of the theorem follow. If any of $ S $, $ rS $ conjugates to $\iota$ and the other one conjugates to $ \gamma_{p}^{m} $ and $ r $ is a primitive element. This implies the proof of $ (c) $.   	
	\qed  \\
	
	\subsection{Proof of the \corref{main1}}
	
	Let $ A $ be a reciprocal element in $ \Gamma_{p} $ such that $ A $ is either symmetric or $ p $-reciprocal.  Let 
	$\tilde{A}= \begin{bmatrix}
		u&&v\\y&&w
	\end{bmatrix} $ be a lift of $ A $ in $ \rm SL(2,\R) $. Then, $ u,v,y,w\in \z[\lambda_p] $. Also, $ uw-vy =1$ and from \propref{Li}, $ (u,v)_{p}=(y,w)_{p}=1. $ Consider,
	\[uw-vy=1 \]
	\[\imp 4 (uw-vy)=4 \]
	\[\imp (u+w)^{2}-(u-w)^{2}-4vy=4.\]
	Let, $ u+w =t$ and $ u-w=b $ and $ d=4vy+b^{2} $. That means
	\[ t^{2}-d=4.\]
	Then, we get a $ 4 $-tuple $ (v,b,y,t) $ in { \pc} such that  $ \begin{bmatrix}
		(t-b)/2&& v\\
		y&& (t+b)/2
	\end{bmatrix}=\tilde{A}. $ 

If $ A $ is represented by some 4-tuple in {\pc}  then there does not exist a 4-tuple in \pc that represents $ \tilde{A}^{-1} $. To see this, let $ (a,b,c,t)\in  $ \pc.  Then there is a symmetric or, a  $ p $-reciprocal element $ B $ such that $\tilde{B}=\begin{bmatrix}
		(t-b)/2&& a\\
		c&& (t+b)/2
	\end{bmatrix}$.  So, $ \tilde{B}^{-1}=\begin{bmatrix}
		(t+b)/2&& -a\\
		-c&& (t-b)/2
	\end{bmatrix} $,  but $ (-a,b,-c,t)\notin $ \pc.   Also, in $ \rm PSL(2,\R) $, $\tilde{B} $ and $ -\tilde{B}$ are same, but $ (a,b,c,-t) \notin  $ \pc.  So,  there are at most half of the symmetric, $ p $-reciprocal or  symmetric $ p $-reciprocal elements that can be given by {\pc}.  Now the assertion follows from  \thmref{sym}. 
	\qed 
	\\


\begin{thebibliography}{4}
		
		
		\bibitem[BR]{BR} M. Baake and J. A. G.  Roberts, 
		\newblock {\em The structure of reversing symmetry groups}.
		\newblock Bull. Austral. Math. Soc. 73 (2006), no. 3, 445--459.
		
		
		\bibitem[DKS]{DKS} B. Demir; \"O. Koruoğlu; R. Sahin, {\em Conjugacy classes of extended generalized Hecke groups}. Rev. Un. Mat. Argentina 57 (2016), no. 1, 49–56. 
		
		\bibitem[FZ]{FZ} W. Feit,  G.  J. Zuckerman, {\it Reality properties of conjugacy classes in spin groups and symplectic groups}, in: Algebraists' Homage: Papers in Ring Theory and Related Topics, New Haven, Conn., 1981, in: Contemp. Math., vol.13, Amer. Math. Soc., Providence, R.I., 1982, pp.239–253.

\bibitem[FK]{fk} R. Fricke, F. Klein, {\it Theorie der Elliptischen Modulfunktionen}, Vol. II, Leipzig, 1892.
		
		\bibitem [HR] {HR} { G. Hoang, (1-FRMC); W. Ressler,(1-FRMC)
			\textit{Conjugacy classes and binary quadratic forms for the Hecke groups}. (English summary)
			Canad. Math. Bull. 56 (2013), no. 3, 570–583.}
		
		
		\bibitem[Ka]{Ka}\ {S. Katok,  {\it Fuchsian groups}. Chicago Lectures in Mathematics. University of Chicago Press, Chicago, IL, 1992. x+175 pp. ISBN: 0-226-42582-7; 0-226-42583-5}
		
		
		\bibitem[La]{l} J. S. W. Lamb,
		\newblock {\it Reversing symmetries in dynamical systems}. 
		\newblock J.  Phys. A 25 (1992), no. 4, 925--937. 	
		
		
		\bibitem[LL] {LL}{ C. Lien Lang (RC-ISU-M); M. Lung Lang,
			\textit{Arithmetic and geometry of the Hecke groups. (English summary)}
			J. Algebra 460 (2016), 392–417.}
		
		\bibitem[LR]{lr} J.S.W.  Lamb and J. A. G. Roberts,  \emph{Time-reversal symmetry in dynamical systems: a survey.}  
		\newblock Time-reversal symmetry in dynamical systems (Coventry, 1996). 
		\newblock Phys. D 112 (1998), no. 1-2,  1–39. 
		
		
		
		\bibitem [MKD]{MKD}{W. Magnus; A. Karrass;D. Solitar \textit{Combinatorial group theory: Presentations of groups in terms of generators and relations}. Interscience Publishers [John Wiley and Sons], New York-London-Sydney, 1966 xii+444 pp.}
		
		
		\bibitem[OS]{os} A.  G. O’Farrell and   I.  Short, {\it Reversibility in Dynamics and Group Theory}. London Mathematical Society Lecture Note Series, vol.416, Cambridge University Press, Cambridge, 2015.
		
		\bibitem[Sa]{Sa} P. Sarnak, 
		\newblock {\em{Reciprocal geodesics}.} 
		\newblock Analytic number theory,  
		Clay Math. Proc. 7,  217–237,  Amer. Math. Soc., Providence, RI, 2007. 
		
		
		
		
		
		
		
		
		
		
		
		
		
		
		
		
	\end{thebibliography}
\end{document}